\documentclass[11pt,letterpaper]{amsart}
\usepackage{amsmath}
\usepackage{amssymb}
\usepackage{amsthm}
\usepackage{hyperref}
\usepackage[noabbrev,capitalize]{cleveref}
\usepackage{mathrsfs}
\usepackage{mathtools}
\usepackage{slashed}
\usepackage{tikz-cd}
\usepackage[shortlabels]{enumitem}

\setenumerate{label=(\roman*)}

\def\sideremark#1{\ifvmode\leavevmode\fi\vadjust{\vbox to0pt{\vss
 \hbox to 0pt{\hskip\hsize\hskip1em
 \vbox{\hsize3cm\tiny\raggedright\pretolerance10000
 \noindent #1\hfill}\hss}\vbox to8pt{\vfil}\vss}}}

\newtheorem{theorem}{Theorem}[section]
\newtheorem{proposition}[theorem]{Proposition}
\newtheorem{lemma}[theorem]{Lemma}

\theoremstyle{definition}
\newtheorem{definition}[theorem]{Definition}

\theoremstyle{remark}
\newtheorem{remark}[theorem]{Remark}

\numberwithin{equation}{section}

\usepackage{ stmaryrd }
\usepackage{import}

\makeatletter
\@namedef{subjclassname@2020}{\textup{2020} Mathematics Subject Classification}
\makeatother

\begin{document}

\title[A splitting theorem for 3-manifold]{A splitting theorem for 3-manifold with nonnegative scalar curvature and mean-convex boundary}

\author{Han Hong}
\address{Department of Mathematics and statistics \\ Beijing Jiaotong University \\ Beijing \\ China, 100044}
\email{hanhong@bjtu.edu.cn}

\author{Gaoming Wang}
\address{Yau Mathematical Sciences Center, Tsinghua University, 100084, Beijing, China}
\email{gmwang@tsinghua.edu.cn}

\begin{abstract}
We show that a Riemannian 3-manifold with nonnegative scalar curvature and mean-convex boundary is flat if it contains an absolutely area-minimizing (in the free boundary sense) half-cylinder or strip. Analogous results also hold for a $\theta$-energy-minimizing half-cylinder, or, under certain topological assumptions, a $\theta$-energy-minimizing strip for $\theta\in (0,\pi)$. 
\end{abstract}
\maketitle

\section{introduction}
The classical Cheeger-Gromoll splitting theorem states that a complete noncompact manifold $M$ with nonnegative Ricci curvature and containing a minimizing geodesic line is isometric to $N\times \mathbb{R}$, where $N$ is complete manifold with nonnegative Ricci curvature. If $M$ has the mean-convex boundary and  contains a minimizing geodesic ray starting from the boundary,  the splitting conclusion fails due to a simple counterexample: the convex domain bounded by a paraboloid in $\mathbb{R}^n$. However, if one requires the boundary to be compact, a similar argument to Cheeger-Gromoll's one can be used to deduce the splitting result. This was proved by A. Kasue \cite[Theorem C]{kasue}. See also C. Croke and B. Kleiner \cite[Theorem 2]{croke-bruce} for a different approach.

A natural direction of extension of Cheeger-Gromoll splitting theorem is to use minimal hypersurface. One question is whether the splitting theorem still holds with area-minimizing hypersurface replacing minimizing geodesic line. In dimension higher than $9$, the answer is negative. There exists a non-flat area-minimizing hypersurface in $\mathbb{R}^n$ for $n\geq 9$, yet Euclidean space is not isometric to a product of such hypersurface with $\mathbb{R}$.
On the other hand, P. Nabonnand \cite{Nabonnand} showed that this is false with dimension greater than $3.$ However, Anderson and Rodriguez \cite{anderson-rodriguez} proved that, if an oriented 3-manifold $M$ has nonnegative Ricci curvature and bounded sectional curvature, moreover, $M$ contains a complete oriented area-minimizing surface, then the universal cover of $M$ is isometric to a product. The boundedness of sectional curvature was removed by Liu \cite{liugang}. This, as a byproduct, proves the Milnor conjecture in dimension 3. 

The doubled Schwarzchild manifold shows that replacing nonnegative Ricci curvature with a weaker curvature assumption, i.e., nonnegative scalar curvature, does not yield the splitting theorem. However, in dimension three, O. Chodosh, M. Eichmair and V. Moraru \cite{chodoshsplittingtheorem} proved a splitting theorem for scalar curvature: Let $M$ be an orientable, complete 3-manifold with nonnegative scalar curvature. If $M$ contains a properly embedded area-minimizing cylinder. Then, up to scaling, a cover of $M$ is isometric to $\mathbb{S}^1\times \mathbb{R}^2.$ Thus they answered a question by M. Cai and G. Galloway \cite{caigalloway}. This is also related to the stability classification of two-sided stable minimal surfaces in 3-manifold with nonnegative scalar curvature. D. Fischer-Colbrie and R. Schoen \cite{Fischer-Colbrie-Schoen-The-structure-of-complete-stable} proved that such surfaces, if noncompact, must be topologically plane or cylinder (see Lemma \ref{fischer-colbrie} below). Note that in Chodosh-Eichmair-Moraru's theorem, area-minimizing can not be replaced by stability. Counterexample was provided in \cite{caigalloway} that $\mathbb{R}\times \mathbb{S}^1$ in $\mathbb{R}\times \mathcal{S}^2$ where $\mathcal{S}^2$ is a smooth flattened sphere (around the equator) is a stable minimal cylinder. 

It is natural to ask whether such splitting result still holds for manifolds with boundary in the presence of minimizing surface with boundary. In fact, we prove
\begin{theorem}\label{maintheorem}
    Let $(M,g)$ be a smooth connected, orientable, complete noncompact Riemannian 3-manifold with nonnegative scalar curvature and mean-convex boundary. Assume that $M$ contains a properly embedded surface $\Sigma$ that is either an absolutely area-minimizing half-cylinder or an absolutely area-minimizing strip in the free boundary sense. Then $(M,g)$ is flat. In fact, up to scaling, a cover of $(M,g)$ is isometric to $\mathbb{S}^1\times \mathbb{R}^2_+$ in the former case, while $(M,g)$ is isometric to $[0,1]\times \mathbb{R}^2$ in the latter case.
\end{theorem}

In the above theorem, a half-cylinder means a topological surface that is homeomorphic to $\mathbb{S}^1\times [0,\infty)$, while a strip means a topological surface that is homeomorphic to $[a,b]\times \mathbb{R}$. Besides, the area-minimzing property in the free boundary sense is defined in Section \ref{sec:energy_minimizing_surfaces}. Geometrically, a surface with such property is a stable free boundary minimal surface. Comparing to the splitting theorem by \cite{kasue,croke-bruce}, a feature of Theorem \ref{maintheorem} is that we do not assume the compactness of the boundary of $M$ and require no information of the number of the boundary components of $M$. 

In fact, we prove a more general result that uses the $\theta$-capillary surface as the splitting surface. To achieve that, we shall define several definitions of minimizing surfaces in Section \ref{sec:energy_minimizing_surfaces}, including absolutely $\theta$-energy-minimizing, homologically $\theta$-energy-minimizing and homologically$^{*}$ $\theta$-energy-minimizing. With the help of these concepts, we are able to show that
\begin{theorem}\label{maintheorem2}
    Let $(M,g)$ be a smooth connected, orientable, complete noncompact Riemannian 3-manifold with nonnegative scalar curvature and mean-convex boundary. Let $\theta\in (0,\pi)$. Assume that $M$ satisfies one of the following conditions:
    
    (1) $M$ contains a properly embedded surface $\Sigma$ that is  an absolutely $\theta$-energy-minimizing half-cylinder.
    
    (2) $M$ contains a properly embedded surface $\Sigma$ that is an absolutely $\theta$-energy minimizing strip whose boundaries lie in two different components of $\partial M$.
    
    Then $(M,g)$ is flat. In fact, up to scaling, a cover of $(M,g)$ is isometric to either
    $$\left\{ (x_1,x_2,x_3)\in \mathbb{S}^1 \times \mathbb{R}^2:x_2\ge \cot \theta x_3 \right\}$$
    or 
    $$\mathbb{S}^1\times[0,1]\times \mathbb{R}.$$
\end{theorem}

Let us emphasis that the absolutely area-minimizing in the free boundary sense can imply $\frac{\pi}{2}$-energy-minimizing. However, these two notions are not equivalent. For precise definitions, see Section \ref{sec:energy_minimizing_surfaces}.


To prove the splitting result, the basic philosophy is to construct a desirable foliation (\cite{anderson-rodriguez,liugang,chodoshsplittingtheorem}) or bifurcation around the base surface (see e.g. \cite{Espinar}). Instead of using implicit function theorem in the case of minimizing surface being compact, one uses the limits of solutions of Plateau-type problems when the base surface is noncompact. Applying perturbation argument one can constrain the convergence sequence and obtain a nonempty smooth limiting surface. In our theorem, when the minimizing surface is a half-cylinder or a long strip, the goal is to construct a sequence of minimizing half-cylinders or minimizing strips. Similar to \cite{chodoshsplittingtheorem}, one of the challenges is that the geometric convergence of solutions of the minimizing problem may produce a limit with a topological type different from the original half-cylinder or strip. Moreover, there are more possibilities when the surface is noncompact and has non-empty boundary. The topology of minimizing surfaces in manifolds of nonnegative scalar curvature and mean-convex boundary can be classified by Fischer-Cobrie and Schoen \cite{Fischer-Colbrie-Schoen-The-structure-of-complete-stable}, Hong and Saturnino\cite{hongcrelle}. Certain topological types, such as plane or cylinder, can be easily ruled out due the presence of the boundary. This differs from Chodosh-Eichmair-Moraru\cite{chodoshsplittingtheorem}. For disc-type and half-plane-type, we can adapt the idea from \cite{chodoshsplittingtheorem}, using energy comparison arguments to exclude them. 

The second challenge is constructing an appropriate minimizing sequence of capillary surfaces. Unlike the case in \cite{chodoshsplittingtheorem}, we cannot use the solution to a Plateau problem with a partial fixed boundary portion, as we currently do not have regularity results near the corner. Instead, we construct a prism-shaped domain and find a minimizer to a certain energy functional in this domain.
Then using the results from \cite{Li2019comparison}, we obtain the existence of the desired minimizer.  

It is worth noting that the extra assumption in (2) of Theorem \ref{maintheorem2} is necessary since the $\theta$-energy functional involves the term $\cos\theta$, which has an indefinite sign, so that the energy comparison no longer works in the strip case. While we believe that this assumption on boundary component is redundant, we do not know how to remove it yet. On the other hand, let us mention that recently there have been many generalizations of splitting results to higher dimensions or complex cases. Interested readers are referred to \cite{EspinarRosenberg,zhu1,zhu2}.

The paper will be arranged as follows.
In Section \ref{background}, we introduce some results on the classification of stable minimal surfaces in both the closed case and the capillary case. In Section \ref{sec:energy_minimizing_surfaces}, we introduce a variety of definitions of minimizing capillary surfaces that are used in our main proof.
In Section \ref{sec:proof} we deal with the proof of the main theorem, Theorem \ref{maintheorem} and Theorem \ref{maintheorem2}.

\subsection{Acknowledgements}
The first author is supported by NSFC No. 1240
-1058 and the Talent
Fund of Beijing Jiaotong University No. 2024XKRC008. The second author is supported by the China Postdoctoral Science Fundation (No. 2024M751604).

\section{Stability of surfaces}\label{background}
In this section, we collect some results on the classification and rigidity of complete stable minimal surfaces in three-dimensional Riemannian manifolds with nonnegative scalar curvature. 

\subsection{Surfaces without boundary}
The following result is due to R. Schoen and S.T. Yau \cite{Schoen-Yau-positive-Ricci-curvature} and D. Fischer-Colbrie and R. Schoen \cite{Fischer-Colbrie-Schoen-The-structure-of-complete-stable}. We use $R_g$ to denote the scalar curvature of the metric $g$ hereafter.
\begin{lemma}
    Let $(M,g)$ be a Riemannian 3-manifold with nonnegative scalar curvature. Let $\Sigma$ be a closed, two-sided, complete stable minimal surface in $M$. Then $\Sigma$ is topologically a sphere or a torus. In the latter case, the surface is totally geodesic and flat. Moreover, the ambient scalar curvature $R_g$ vanishes along $\Sigma.$
\end{lemma}

The following result is due to D. Fischer-Colbrie and R. Schoen\cite[Theorem 3]{Fischer-Colbrie-Schoen-The-structure-of-complete-stable} and C. Carlotto and O. Chodosh and M. Eichmair\cite{Carlotto-Chodosh-Eichmair-PMT} (see also \cite{Espinar}).
\begin{lemma}\label{fischer-colbrie}    Let $(M,g)$ be a Riemannian 3-manifold with nonnegative scalar curvature. Let $\Sigma$ be a noncompact, two-sided, complete stable minimal surface in $M$. Then $\Sigma$ is conformally diffeomorphic to a plane or a punctured plane (cylinder). In the latter case, the surface is totally geodesic and flat. Moreover, the ambient scalar curvature $R_g$ vanishes along $\Sigma.$
\end{lemma}

The following rigidity result is due to M. Cai and G. Galloway\cite[Theorem 2]{caigalloway}.
\begin{lemma}
    Let $(M,g)$ be a complete connected 3-manifold of nonnegative scalar curvature whose boundary (possibly empty) is mean convex. If $\Sigma$ is a two-sided torus which has the least area in its isotopy class, then $(M,g)$ is flat.
\end{lemma}
The rigidity results regarding stable sphere were studied by H. Bray, S. Brendle and A. Neves\cite{braybrendleneves}. See also related rigidity works by I. Nunes \cite{nunesrigidity} and M. Micallef and V. Moraru \cite{micallefmoraru}.

\subsection{Surfaces with boundary}\label{capillarysurfaces}

Let $(M,g)$ be a smooth orientable Riemannian 3-manifold with nonempty boundary. Let $\Sigma$ be a smooth surface with nonempty boundary and let $\varphi:\Sigma\rightarrow M$ be a two-sided proper embedding such that $\varphi(\Sigma)\cap \partial M=\varphi(\partial\Sigma).$ Let us just denote the image by $\Sigma$ for simplicity. Fix a unit normal vector $N$ along $\Sigma$, and denote $\nu$ the unit conormal vector of $\partial\Sigma$ in $\Sigma$. Denote $\bar{N}$ the outward unit normal vector of $\partial M$ in $M$, and choose $\bar{\nu}$ to be the unit normal vector of $\partial \Sigma$ in $\partial M$ such that $\{\bar{N},\bar{\nu}\}$ and $\{N,\nu\}$ determines the same orientation in $(T\partial\Sigma)^\perp.$

Let $X(\cdot,t):\Sigma\times (-\epsilon,\epsilon)\rightarrow M$ be a compactly supported variation such that $X(\partial \Sigma,t)\subset \partial M$ for all $t \in (-\epsilon, \epsilon)$. Let $\Omega \subset \Sigma$ be the support of the variation $X(\cdot, t)$. With respect to this variation $X$, we consider following two functionals:
\begin{align*}
	A(t)={}&\int_\Omega\ dA_t,\\
	W(t)={}&\int_{[0,t]\times(\partial\Sigma \cap \Omega)}X^{*}\ dA_{\partial M},
\end{align*}
where $dA_t$ is the area element with respect to the metric on $X(\Sigma,t)$, $d A_{\partial M}$ is the area element of $\partial M$, and
$X^*(dA_{\partial M})$ is the pullback area element. $W(t)$ is usually called wetting area functional and geometrically means the signed area of the domain bounded by $\partial\Sigma\cap \Omega$ and $\partial\Sigma_t\cap\Omega.$ 

Fix a number $\theta\in(0,\pi)$, and consider the $\theta$-energy functional
\[E_\theta(t)=A(\Sigma_t)-\cos\theta W(t).\]
We call a surface $\theta$-capillary minimal surface if it is a critical point of the $\theta$-energy functional for any compactly supported variation.  Notice that a $\theta$-capillary minimal surface intersects $\partial M$ at a constant angle $\theta$, i.e., $g(\nu,\bar{\nu})=\theta$. In particular, if $\theta=\pi/2$, the energy functional reduces to the usual area functional, and the critical point corresponds to a free boundary minimal surface. By calculating the second derivative of $E_\theta$ with respect to $t$ and evaluating at $t=0$, we have
\[E''(0)=\int_\Sigma|\nabla\varphi|^2-(\operatorname{Ric}_M(N,N)+|A_\Sigma|^2)\varphi^2-\int_{\partial\Sigma}q\varphi^2\]
where 
$q=\frac{1}{\sin\theta}h_{\partial M}(\bar{\nu},\bar{\nu})+\cot\theta h_{\Sigma}(\nu,\nu).$ Here $h_{\partial M}$ and $h_\Sigma$ are the second fundamental forms of $\partial M$ in $M$ and $\Sigma$ in $M$ with respect to $\bar{N}$ and $N$, respectively.

The capillary minimal surface $\Sigma$ is called stable if $E''(0)\geq 0$ for any compactly supported function $\varphi$. In Section \ref{sec:energy_minimizing_surfaces}, we provide various definitions of $\theta$-energy-minimizing of a surface, and they all imply that the surface is a critical point and a local minimizer of $E_\theta$ with respect to any compact variation $X$, thus is a stable capillary minimal surface.

Assume that $\Sigma$ is a compact surface. The following result is due to E. Longa \cite[Theorem B]{longacapillary}. The $\theta=\pi/2$ case was proved by L. Ambrozio \cite[Theorem 7]{ambroziorigidity}.
\begin{lemma}\label{ambrozio}
    Let $(M,g)$ be a Riemannian 3-manifold with mean-convex boundary. Suppose $\Sigma$ is a compact, properly embedded, two-sided, $\theta$-energy minimizing surface such that $I_\theta(\Sigma)=2\pi\chi(\Sigma)$. Assume that 
     \vskip.1cm
     \begin{enumerate}[label=\textnormal{(\theenumi)}]
	 	\item each component of $\partial\Sigma$ is locally length-minimizing in $\partial M$; or
\item $\inf H_{\partial M}=0$.
    \end{enumerate}
    \vskip.1cm
    Then either $\theta=\pi/2$ or $(M,g)$ is flat and totally geodesic cylinder, $M$ is flat and $\partial M$ is totally geodesic around $\Sigma.$ In the first case, there is a neighborhood $\Sigma$ in $M$ that is isometric to $(\Sigma,(-\epsilon,\epsilon), g_\Sigma+dt^2)$ where $(\Sigma, g_\Sigma)$ has constant Gaussian curvature $\frac{1}{2}\inf R$ and $\partial\Sigma$ has constant geodesic curvature $\inf H_{\partial M}$ in $\Sigma.$
    
\end{lemma}

Here $I_\theta(\Sigma)=\frac{1}{2}\inf R |\Sigma|+\frac{1}{\sin\theta}\inf H_{\partial M}|\partial\Sigma|$, and $\theta$-energy minimizing means that for all variation $\Sigma_t$ of $\Sigma$, we have $E_\theta(t)\ge E_\theta(0)$.
It is easy to see that $I_\theta(\Sigma)\leq 2\pi\chi(\Sigma)$ is always true for compact stable $\theta$-capillary minimal surfaces (see \cite[Theorem A]{longacapillary}). Thus, if the ambient scalar curvature is nonnegative and the boundary is mean convex, then $\chi(\Sigma)\geq 0$, that is, $\Sigma$ is topologically a disc or an annulus. Properly embeddedness means that the interior of $\Sigma$ does not intersect $\partial M.$  The a priori assumptions (1) and (2) are important in the proof. In fact, the topology of annulus guarantees $(2).$
\begin{lemma}
\label{lemma:flatAnnulus}
    Let $(M,g)$ be a Riemannian 3-manifold with mean-convex boundary and nonnegative scalar curvature. If $\Sigma$ is a compact, properly embedded, two-sided, $\theta$-energy minimizing annulus. Then $(M,g)$ is  flat and $\partial M$ is totally geodesic around $\partial\Sigma$. In particular, if $\theta=\pi/2$, $(M,g)$ is isometric to $((-\epsilon,\epsilon)\times \Sigma, dt^2+g_{\Sigma})$ where $(\Sigma,g_\Sigma)$ is a flat annulus and $\partial\Sigma$ is geodesic in $\Sigma$.  If $\theta\neq\pi/2 $, $(M,g)$ is locally isometric to a wedge domain shown in Figure \ref{fig:wedge}.
\end{lemma}

\begin{figure}[ht]
    \centering
	\begingroup
	\def\svgwidth{0.6\columnwidth}
\begingroup%
  \makeatletter%
  \providecommand\color[2][]{%
    \errmessage{(Inkscape) Color is used for the text in Inkscape, but the package 'color.sty' is not loaded}%
    \renewcommand\color[2][]{}%
  }%
  \providecommand\transparent[1]{%
    \errmessage{(Inkscape) Transparency is used (non-zero) for the text in Inkscape, but the package 'transparent.sty' is not loaded}%
    \renewcommand\transparent[1]{}%
  }%
  \providecommand\rotatebox[2]{#2}%
  \newcommand*\fsize{\dimexpr\f@size pt\relax}%
  \newcommand*\lineheight[1]{\fontsize{\fsize}{#1\fsize}\selectfont}%
  \ifx\svgwidth\undefined%
    \setlength{\unitlength}{383.59430749bp}%
    \ifx\svgscale\undefined%
      \relax%
    \else%
      \setlength{\unitlength}{\unitlength * \real{\svgscale}}%
    \fi%
  \else%
    \setlength{\unitlength}{\svgwidth}%
  \fi%
  \global\let\svgwidth\undefined%
  \global\let\svgscale\undefined%
  \makeatother%
  \begin{picture}(1,0.55694468)%
    \lineheight{1}%
    \setlength\tabcolsep{0pt}%
    \put(0,0){\includegraphics[width=\unitlength,page=1]{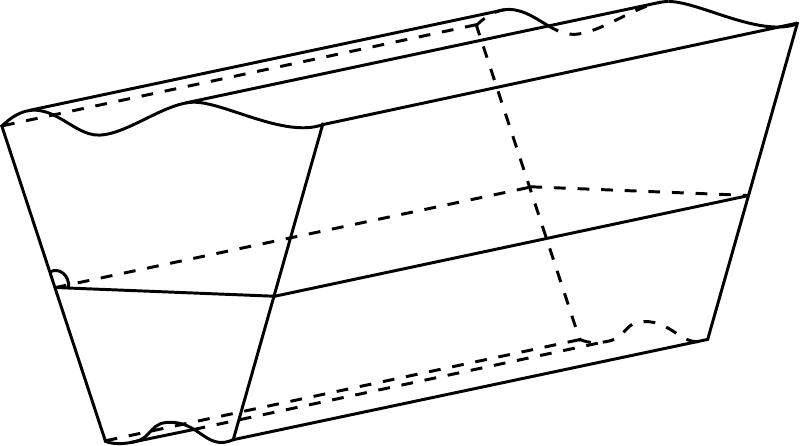}}%
    \put(0.07684909,0.22901359){\makebox(0,0)[lt]{\lineheight{1.25}\smash{\begin{tabular}[t]{l}$\theta$\end{tabular}}}}%
    \put(0.13120851,0.51077871){\makebox(0,0)[lt]{\lineheight{1.25}\smash{\begin{tabular}[t]{l}$M$\end{tabular}}}}%
    \put(0.17814305,0.14082998){\makebox(0,0)[lt]{\lineheight{1.25}\smash{\begin{tabular}[t]{l}$\Sigma$\end{tabular}}}}%
    \put(-0.01508581,0.10693678){\makebox(0,0)[lt]{\lineheight{1.25}\smash{\begin{tabular}[t]{l}$\partial M$\end{tabular}}}}%
  \end{picture}%
\endgroup%

	\endgroup

	\caption{A locally wedge domain}
    \label{fig:wedge}
\end{figure}
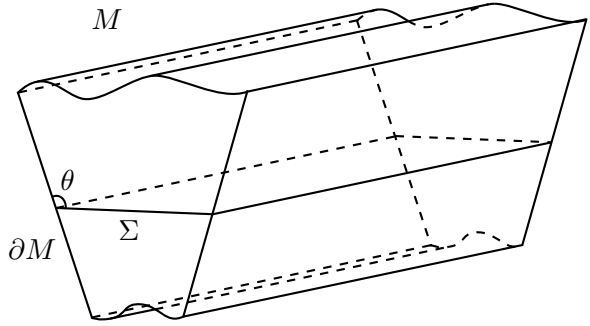

For completeness and convenience of readers, we include the proof briefly here.

\begin{proof}[Proof of Lemma \ref{lemma:flatAnnulus}]
Plugging $\varphi=1$ in the stability inequality and using Gauss equation show that
\[0\leq \int_\Sigma -\frac{1}{2}R+K_\Sigma -\frac{1}{2}|A|^2+\int_{\partial\Sigma}k_{\partial\Sigma}-\frac{1}{\sin\theta}H_{\partial M}.\]
The assumption and Gauss-Bonnet formula imply that
\[0\leq 2\pi(2-2\gamma-k).\]
When the surface is annulus, i.e., $\gamma=0,\ k=2$, we achieve all the equalities in the deduction. In particular, $H_{\partial M}=0$ along $\partial \Sigma.$

Using the implicit function theorem, one can construct a foliation of $\theta$-capillary CMC annulus around $\Sigma$ (See \cite[Proposition 3.2]{longacapillary}). Following the notations from \cite{longacapillary}, we denote the parametrization of the foliation $\{\Sigma_t\}_{t\in(-\epsilon,\epsilon)}$ by $G$ and denote the unit normal vector of each leaf by $N_t$. It is shown in \cite[Pages 17-18]{longacapillary} that $\Sigma_t$ is flat annulus and $N_t$ is parallel vector field. Since each leaf $\Sigma_t$ is a flat annulus, up to a dilation of ambient metric, we assume that $\Sigma=\mathbb{S}^1(l_1)\times [0,l_2]$ and $\Sigma_t=\mathbb{S}^1(l_1(t))\times [0,l_2(t)]$. Because $N_t$ is parallel, $\Sigma_s$ contains a domain that is isometric to $\Sigma_t$ for $s>t$. Thus we can assume that $l_1(t)=l_1$ for $t\in (-\epsilon,\epsilon).$ Notice that $|\Sigma_t|=l_1\cdot l_2(t)$. On the other hand,
\[\frac{d|\Sigma_t|}{dt}=\int_{\Sigma_t}H_tf_t +\int_{\partial\Sigma_t}g(\frac{\partial G}{\partial t},\nu_t)=\cos\theta|\partial\Sigma_t|\]
where $\nu_t$ is the conormal of $\partial\Sigma_t$ in $\Sigma_t$ and $f_t=g(\frac{\partial G}{\partial t},N_t).$ Then
\[\frac{d l_2(t)}{dt}=2\cos\theta.\]
Thus $l_2(t)=l_2+2t\cos\theta.$ In particular, when $\theta=\pi/2$, all sheets are isometric to each other and the normal vector field of $\Sigma_t$ is a local vector field on $M$ that induces a flow by isometries and provides a local splitting. This completes the proof.
\end{proof}

When the surface $\Sigma$ under consideration is not compact, we have the following result due to the first-named author and A. Saturnino \cite[Corollary 4.5]{hongcrelle}. The flatness and the geodesic property of the boundary are not asserted in Corollary 4.5 of \cite{hongcrelle}, but it can be deduced by following the proof there.

\begin{lemma}
\label{lemma:flatnessBdry}
    Let $(M,g)$ be a Riemannian 3-manifold with mean-convex boundary and nonnegative scalar curvature. Let $\Sigma$ be a complete noncompact, properly immersed, two-sided stable capillary minimal surface in $M$. Then $\Sigma$ is conformally diffeomorphic to a strip, or a half-cylinder, or a half-plane. In the former two cases, the surface is totally geodesic and flat. Moreover, the ambient scalar curvature vanishes along $\Sigma$, $H_{\partial M}$ vanishes along $\partial \Sigma$ and $\partial\Sigma$ is geodesic in $\Sigma.$  
\end{lemma}

In the cases of half-cylinder or strip, the Gauss equation implies that the ambient Ricci tensor evaluated in the normal direction of $\Sigma$ vanishes along $\Sigma.$

\section{Energy-minimizing surfaces}
\label{sec:energy_minimizing_surfaces}

We define and collect some definitions and properties of energy-minimizing surfaces relevant to the proof of the main theorem.
Fix a smooth complete Riemannian $n$-manifold $(M,g)$ with non-empty boundary $\partial M$ (possibly disconnected) and a real value $\theta \in (0,\pi)$.
Nash's embedding theorem allows us to assume that $M$ is isometrically embedded into $\mathbb{R}^N$ for some $N$. Thus, currents on $M$ can be defined (refer to \cite[Chapter 6]{simon1983lectures} for more details).
For any $k$-dimensional oriented submanifold $\Sigma$ in $M$, $\llbracket \Sigma \rrbracket$ denotes the $k$-dimensional integral current in $M$ induced by $\Sigma$ and its orientation.

We divide the components of $\partial M$ into two parts: $\partial^s M$ and $\partial M \backslash \partial^s M$.
For any $k$-dimensional integral current $T$ in $M$, its boundary is separated into $\partial^s T := \partial T \lfloor (\partial^s M)$ and $\partial^i T := \partial T - \partial T \lfloor (\partial M)$.
Consequently, for any hypersurface $\Sigma$ in $M$, we define $\partial^s \Sigma$ and $\partial^i \Sigma$ such that
\[
	\llbracket \partial^s \Sigma\rrbracket=\partial^s \llbracket \Sigma\rrbracket,\quad \llbracket \partial^i \Sigma\rrbracket=\partial^i \llbracket \Sigma\rrbracket.
\]
Similarly, for any relatively open subset $U \subset M$, $\partial^s U$ and $\partial^i U$ are defined in the same manner.



\begin{definition}
    For a hypersurface $\Sigma$ in $M$ with boundary $\partial \Sigma \subset \partial M$,
we say that $\Sigma$ is \textit{absolutely $\theta$-energy-minimizing} if, for any pre-compact (relatively) open subset $U\subset M$, and any hypersurface $\Sigma'$ with $\Sigma'=\Sigma$ in $M\setminus U$, and $\partial\Sigma'$ can be written as
\[
\partial \llbracket\Sigma\rrbracket=\partial \llbracket \Sigma'\rrbracket+\partial W,
\]
where $W$ is an $n$-dimensional integral current supported in $U \cap \partial M$, it holds that
\[
	|\Sigma\cap U|\le |\Sigma'\cap U|+\cos \theta |W|,
\]
where $|W|$ denotes the signed area of $W$.
Precisely, let $\omega_{\partial M}$ to be the volume form of $\partial M$, then we define $|W|=W(\omega_{\partial M})$.
\end{definition}

We say that $\Sigma$ is \textit{absolutely area-minimizing in the free boundary sense} if for any pre-compact (relatively) open subset $U\subset M$, and any hypersurface $\Sigma'$ with $\Sigma'=\Sigma$ in $M\setminus U$,
we have
\[
	|\Sigma\cap U|\le |\Sigma'\cap U|.
\]
If $U\cap \partial M$ is empty in above definition, then it reduces to the usual definition of absolutely area-minimizing in the interior. Note that in general, the absolutely $\frac{\pi}{2}$-energy-minimizing property does not imply the absolutely area-minimizing property in the free boundary sense.



We shall further introduce a few definitions that will be useful in our proof. From now on, we will fix $\partial^s M$ and assume that $\partial\Sigma\subset \partial^s M$. Moreover, we will use open subset to mean relatively open subset. 

\begin{definition}
    We say that $\Sigma$ is \textit{homologically $\theta$-energy-minimizing} if for any pre-compact open subset $U\subset M$ and any hypersurface $\Sigma'$ with $\partial\Sigma'\subset \partial^sM$ such that
\[
	\llbracket \Sigma\rrbracket=\llbracket \Sigma'\rrbracket+\partial^iT,
\]
where $T$ is an $(n+1)$-dimensional integral current which has compact support in $U$, we have
\[
	|\Sigma\cap U|\le|\Sigma'\cap U|+\cos \theta |\partial^s T|.
\]
\end{definition}



\begin{definition}
    Given an open subset $U\subset M$, we say that $\Sigma$ is \textit{homologically$^*$ $\theta$-energy-minimizing} in $U$ if, for any pre-compact open subset $V\subset\subset U$ and any hypersurface $\Sigma'$ with $\partial\Sigma'\cap U\subset \partial^sM$ and
\[
	\llbracket \Sigma\rrbracket=\llbracket \Sigma'\rrbracket+\partial^i T,
\]
where $T$ is an $(n+1)$-dimensional integral current with $\partial^sT$ and $\partial^iT$ having compact supports in $V$, the following inequality holds:
\[
	|\Sigma\cap V|\le|\Sigma'\cap V|+\cos \theta |\partial^sT|.
\]
\end{definition}

In particular, if $U=M$, we say that $\Sigma$ is homologically$^*$ $\theta$-energy-minimizing in $M$.
At last, given an open subset $U\subset M$, we say that $\Sigma$ is \textit{homologically$^*$ area-minimizing in the free boundary sense} in $U$ if, for any pre-compact open subset $V\subset\subset U$ and any $\Sigma'$ such that
\[
	\llbracket \Sigma\rrbracket=\llbracket \Sigma'\rrbracket+\partial^i T,
\]
where $T$ is an $(n+1)$-dimensional integral current with $\partial^iT$ having compact support in $V$, we have
\[
	|\Sigma\cap V|\le|\Sigma'\cap V|.
\]

Under the above setting, we now prove two propositions.
\begin{proposition}
	\label{prop:homologically_energy_minimizing}
	If $\Sigma$ is homologically$^*$ $\theta$-energy-minimizing in $M$, then each component of $\Sigma$ is homologically$^*$ $\theta$-energy-minimizing in $M$.
\end{proposition}
\begin{proof}
	Let $\Sigma_0$ be the component of $\Sigma$.
Fix a pre-compact open subset $U$ of $M$, and consider any $(n+1)$-dimensional integral current $T$ such that
\[
	\llbracket \Sigma_0\rrbracket=\llbracket \Sigma_0'\rrbracket+\partial^iT,\quad \text{ and }\quad \mathrm{spt}(\partial^i T+\partial^sT)\subset U,
\]
for some hypersurface $\Sigma_0^{'}$.
Note that we can also write it as
\[
	\llbracket \Sigma\rrbracket=\llbracket \Sigma_0'\rrbracket+\llbracket \Sigma\backslash \Sigma_0\rrbracket+\partial^iT.
\]
Since $\Sigma$ is homologically$^*$ $\theta$-energy-minimizing in $M$, we have
\[
	|\Sigma\cap U|\le |\Sigma_0'\cap U|+|(\Sigma\backslash \Sigma_0)\cap U|+\cos \theta |\partial^sT|.
\]
Therefore, it follows that:
\[
	|\Sigma_0\cap U|\le |\Sigma_0'\cap U|+\cos \theta |\partial^sT|.
\]
Thus we complete the proof.
\end{proof}

\begin{proposition}
	
	\label{prop_convergence}
	Let $U_i\subset \subset M$ be a sequence of pre-compact open subsets of $M$ such that $\cup_i U_i=M$.
	Suppose $\Sigma_i$ is a sequence of homologically$^*$ $\theta$-energy-minimizing hypersurfaces in $U_i$ and $\Sigma_i$ converges to $\Sigma$ locally and smoothly.
	Then, $\Sigma$ is homologically$^*$ $\theta$-energy-minimizing in $M$.
\end{proposition}
\begin{proof}
	Without loss of generality, we may assume that $U_i\subset U_{i+1}$ for all $i$.
	Given an $(n+1)$-dimensional integral current $T$ such that
	\[
		\llbracket \Sigma\rrbracket=\llbracket \Sigma'\rrbracket+\partial^i T,\quad \text{ and }\quad \mathrm{spt}(\partial^iT+\partial^s T)\subset U,
	\]
	for some $U\subset \subset M$, and some hypersurface $\Sigma'$.
    We consider $i$ large enough such that $U\subset U_i$.
	Note that the homologically$^*$ $\theta$-energy-minimizing can pass to the limit in $U$, by adapting the proof of the area-minimizing can be passed to the limit in $U$ for currents (see \cite[Theorem 34.5]{simon1983lectures}).
	Then, we have
	\[
		|\Sigma\cap U|\le |\Sigma'\cap U|+\cos \theta |\partial^s T|.
	\]
	This completes the proof.
\end{proof}

\section{Proof of Main Theorem}\label{sec:proof}


In the following we lift the ambient manifold $M$ to a cover such that the inclusion map of the surface induces a surjection at the level of fundamental group.

\begin{lemma}\label{cover}
    Let $(M,g)$ be an orientable, Riemannian 3-manifold with noncompact complete boundary and $\Sigma\subset M$ be a properly embedded, orientable, absolutely $\theta$-energy-minimizing surface with boundary in $\partial M$. There exists a covering $p:\tilde{M}\rightarrow M$ such that $p_*(\pi_1(\tilde{M}))=i_*(\pi_1(\Sigma))$ where $i$ is the inclusion map of $\Sigma.$ The inclusion map lifts to a proper embedding $\tilde{i}: \Sigma\rightarrow \tilde{M}$ and $\tilde{i}_*: \pi_1(\Sigma)\rightarrow \pi_1(\tilde{M})$ is surjective. Moreover, $\tilde{\Sigma}=\tilde{i}(\Sigma)\subset \tilde{M}$ is absolutely $\theta$-energy-minimizing in $ (\tilde{M},\tilde{g})$ where $\tilde{g}=p^* g.$
\end{lemma}

\begin{proof}
    Suppose $\tilde{\Sigma}$ is not absolutely $\theta$-energy-minimizing. Consider a competing surface in $\tilde{M}$ that shares the same interior boundary of a compact subset of $\tilde{\Sigma}$. Projecting downstairs and using cut and past argument we obtain a properly embedded competitor, contradicting the $\theta$-energy-minimizing property of $\Sigma$ in $M.$
\end{proof}
\begin{remark}
    Note that above lemma still holds if we replace the absolutely $\theta$-energy-minimizing by the absolutely area-minimizing in the free boundary sense.
\end{remark}


We now give a proof of the main theorem.
\begin{proof}
	[Proof of Theorem \ref{maintheorem} and Theorem \ref{maintheorem2}]

	Let $S \subset M$ be a properly embedded half-cylinder or a flat strip that is absolutely $\theta$-energy-minimizing or absolutely area-minimizing in the free boundary sense.
	If $S$ is a flat strip and $S$ is absolutely $\theta$-energy-minimizing, we further assume that each component of $\partial S$ belongs to two different boundary components of $\partial M$. This means that the number of boundary components of $M$ is at least two.
    
        In this proof, we mainly focus on the absolutely $\theta$-energy-minimizing case.
        Most of the arguments can be straightforwardly adapted to the case where $S$ is absolutely area minimizing in the free boundary sense, expect for the case where $S$ is a strip, which we will discuss in more detail.

	By Lemma \ref{lemma:flatnessBdry}, $S$ is totally geodesic, flat and its boundary is geodesic. Scaling allows us to assume that $S$ is isometric to either the standard half-cylinder $\mathbb{S}^1 \times \mathbb{R}^+$ or the flat strip $[0,a] \times \mathbb{R}$ for some $a>0$, depending on the case.

	Using Lemma \ref{cover}, we may further assume that the inclusion map $i: S \to M$ induces a surjection $i_*: \pi_1(S) \to \pi_1(M)$. In particular, if $S$ is a strip, then $\pi_1(M)$ is trivial.

	If $S$ is separating, then $M\backslash S$ has two components. Otherwise, it is connected after cutting. We choose $\hat{M}$ to be one of the components of $M \setminus S$. Note that the boundary of $\hat{M}$ consists of two parts:
	$\partial \hat{M} \cap \partial M$ and $\partial \hat{M} \setminus \partial M$.
   We denote $\partial^s \hat{M}$ to be the union of the components of $\partial \hat{M} \cap \partial M$ that has non-empty intersection with $\partial S$. Then $\partial \hat{M}\backslash \partial M$ contains one component or two components (if $S$ is non-separating). In either case, we choose one and denote it by $\Sigma$. Apparently, $\Sigma$ is isometric to $S$.
	In all, we have a new manifold $(\hat{M},\hat{g})$ whose boundary can be divided into following three parts:
    \begin{itemize}
        \item $\partial^s \hat{M}$.
        \item $\Sigma$, or $\Sigma\cup \Sigma.$
        \item $\partial^e\hat{M}:=(\partial\hat{M}\cap \partial M)\backslash \partial^s \hat{M}$.
    \end{itemize} Let $\partial^s$ to be the operator defined in the previous section with respect to $\partial^s \hat{M}$, so
     we have that $\partial^iT=\partial T\lfloor(\hat{M}\backslash \partial M)$ for any $k$-dimensional integral current $T$, $k=2,3$.
	Since $\Sigma$ is absolutely $\theta$-energy-minimizing, the contact angle between $\Sigma$ and $\partial^s \hat{M}$ is either $\theta$ or $\pi-\theta$, depending on the orientation of $\Sigma$.
	Without loss of generality, we assume that the contact angle is $\theta$, and the orientation of $\Sigma$ agrees with the orientation induced from $\hat{M}$ and the outward normal vector field of $\Sigma$.




	We denote by $\Sigma_h$ the portion of $\Sigma$ corresponding to $\mathbb{S}^1 \times [0,h]$ if $S$ is a half-cylinder, or $[0,a] \times [-h,h]$ if $S$ is a strip. The curve $\gamma$ on $\Sigma$ corresponds to $\mathbb{S}^1 \times \{1\}$ if $S$ is a half-cylinder, or $[0,a] \times \{0\}$ if $S$ is a strip.
	Choose a point $p$ on $\Sigma$ that corresponds to $(0,1)$ if $S$ is a half-cylinder, or $(a/2,0)$ if $S$ is a strip. Fix a unit-speed geodesic $c: [0,\varepsilon) \to \hat{M}$ with $c(0) = p$ and $\dot{c}(0) \perp T_{c(0)} \Sigma$. Notice that since $\Sigma$ is minimal and $\partial M$ is mean convex, $c(t)$ does not intersect $\partial \hat{M}$ for $t\in[0,\varepsilon)$ when $\varepsilon$ is small. According to the metric perturbation construction in \cite{Carlotto-Chodosh-Eichmair-PMT}, for $\varepsilon < 1/10$, there exists a family of smooth Riemannian metrics $\{ \hat{g}(r,t) \}_{r,t \in (0,\varepsilon)}$ on $\hat{M}$ such that:
	\begin{enumerate}
	    \item $\hat{g}(r,t) \to \hat{g}$ in $C^3$ as $t, r \to 0^+$.
	    \item $\hat{g}(r,t) \to \hat{g}$ smoothly as $t \to 0^+$ for each fixed $r \in (0,\varepsilon)$.
	    \item $\hat{g}(r,t) = \hat{g}$ on $\{x \in \hat{M} : \mathrm{dist}_{\hat{g}}(x, c(2r)) \geq 3r \}$.
	    \item $\hat{g}(r,t) < \hat{g}$ as quadratic forms in $\{x \in \hat{M} : \mathrm{dist}_{\hat{g}}(x, c(2r)) < 3r \}$.
	    \item $\hat{g}(r,t)$ has positive scalar curvature in $\{x \in \hat{M} : r < \mathrm{dist}_{\hat{g}}(x, c(2r)) < 3r \}$.
	    \item $\hat{M}$ is weakly mean-convex with respect to $\hat{g}(r,t)$.
	\end{enumerate}

	The next step is to find a compact minimizer of the appropriate $\theta$-energy in some pre-compact open subset containing $\Sigma_h$, which differs from $\Sigma_h$ for a fixed $h>2$.
	There might be some regularity issues near $\partial \Sigma_h \backslash \partial^s \hat{M}$. To overcome this, we construct the domain in the following manner. This is inspired by the polyhedron comparison theorem of Li \cite{Li2019comparison}.

	Let $h\ge 2$ be an integer and $0<\delta<\frac{1}{10}$, we consider a small wedge-shaped neighborhood $U_h^\delta$ of $\Sigma_h$ in $\hat{M}$ described as follows:
    
	The boundary of $U_h^\delta$ consists of four parts: $\Sigma_h$, $\partial^s U_h^\delta=\partial U_h^\delta\cap \partial^s\hat{M}, \partial^fU_h^\delta, \partial^t U_h^\delta$. Each part is smooth and satisfies the following conditions:
	\begin{itemize}
		\item $\measuredangle _g (\Sigma_h, \partial^s U_h^\delta) = \theta$, $\measuredangle _g(\partial^t U_h^\delta, \partial^s U_h^\delta)=\pi-\theta$.
		\item $\measuredangle _g (\Sigma_h, \partial^f U_h^\delta)=\measuredangle _g(\partial^t U_h^\delta, \partial^f U_h^\delta)=\measuredangle _g(\partial^s U_h^\delta, \partial^f U_h^\delta)=\frac{\pi}{2}$.
		\item $\partial^tU_h^\delta$ is disjoint from $\Sigma_h$.
		\item The area of $\partial^f U_h^\delta$ is less than $\delta$, and the distance from any point in $\partial^f U_h^\delta$ to $\Sigma_h \cap \partial^f U_h^\delta$ is less than $\delta$.
    \end{itemize}
    \begin{figure}[ht]
        \centering
	\begingroup
	\def\svgwidth{0.8\columnwidth}
	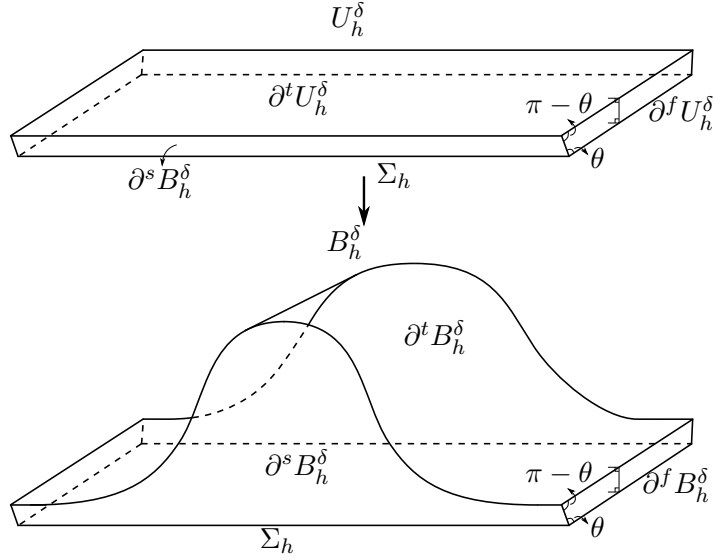
	\endgroup

    	\caption{The choice of $U_h^\delta$ and $B_h^\delta$ when $S$ is a strip}
        \label{fig:domains}
    \end{figure}
    \begin{figure}[ht]
        \centering
	\begingroup
	\def\svgwidth{0.8\columnwidth}
	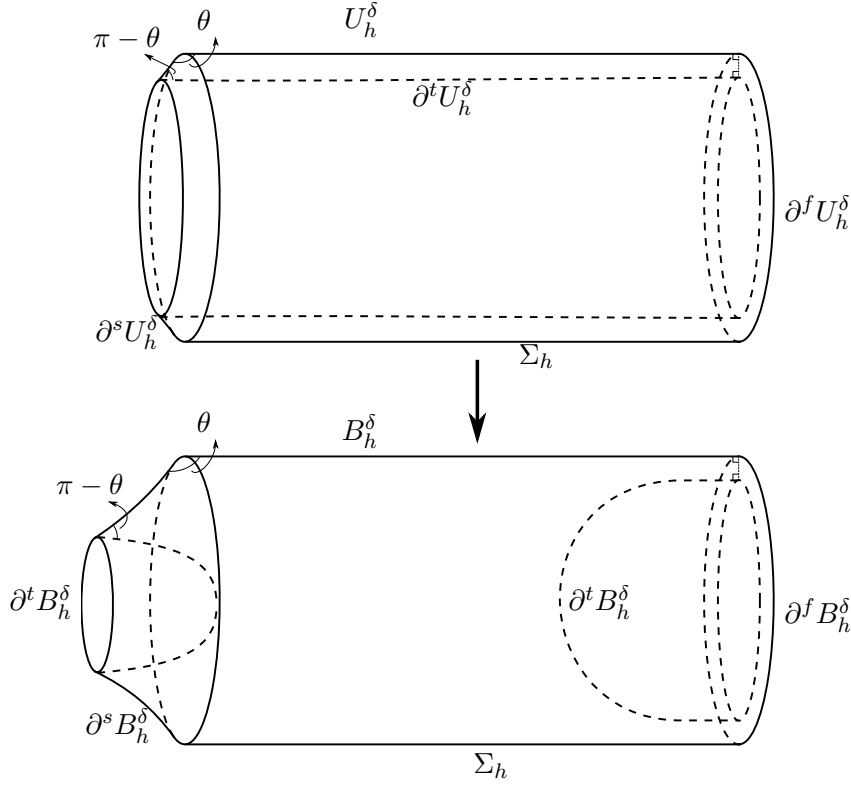
	\endgroup

        \caption{The choice of $U_h^\delta$ and $B_h^\delta$ when $S$ is a half-cylinder}
        \label{fig:half-cylinderwedge}
    \end{figure}
    See Figure \ref{fig:domains} and Figure \ref{fig:half-cylinderwedge} for the shape of $U_h^\delta$ in the cases of half-cylinder and strip, respectively. 
	Indeed, $U_h^\delta$ looks like a curved-prism, with $\Sigma_h$ as the bottom face and $\partial^t U_h^\delta$ as the top face, and it has side faces $\partial^s U_h^\delta$ and $\partial^f U_h^\delta$.
	Note that $\partial^s U_h^\delta$ is disjoint from $\partial^f U_h^\delta$ when $S$ is a half-cylinder.

    Next, fix two small neighborhood $U_{1h}\subset U_{2h}$ of $\partial\Sigma\backslash \partial^s\hat{M}$.
    For instance, we can define 
    $$U_{1h}=\{ \mathrm{dist}_{\hat{g}}(x,\partial\Sigma_h\backslash \partial^s \hat{M})<\frac{1}{9} \}\text{ and }U_{2h}=\{ \mathrm{dist}_{\hat{g}}(x,\partial\Sigma_h\backslash \partial^s \hat{M})<\frac{1}{8} \}.$$
	We then choose $B_h^\delta$ to be a pre-compact open subset of $\hat{M}$, whose boundary consists of five parts: $\Sigma_h$, $\partial^s B_h^\delta=\partial B_h^\delta \cap \partial^s\hat{M}$, $\partial^fB_h^\delta$, $\partial^tB_h^\delta$ and $\partial^e B_h^\delta=B_h^\delta\cap \partial^e\hat{M}$. Each part is smooth, and they satisfy the following conditions:
	\begin{itemize}
		\item $U_h^\delta\subset B_h^\delta$.
		\item $B_h^\delta$ coincides with $U_h^\delta$ in $U_{1h}$. In particular, $\partial^fB_h^\delta=\partial^f U_h^\delta$.
		\item $B_h^\delta$ coincides with $B_h^{\delta'}$ outside $U_{2h}$ for any $0<\delta'<\delta<\frac{1}{10}$.
		\item $B_{h}^\delta\cap (\hat{M}\backslash U_{2h})\subset B_{h'}^\delta\cap (\hat{M}\backslash U_{2h'})$ for $h< h'$.
            \item $\bigcup_{h\ge 2}B_h^\delta\cap (\hat{M}\backslash U_{2h})=\hat{M}$.
		\item $\measuredangle_{\hat{g}}(\Sigma_h, \partial^s B_h^\delta)=\theta$, $\measuredangle_{\hat{g}}(\partial^t B_h^\delta, \partial^s B_h^\delta)=\pi-\theta$.
		\item $\measuredangle_{\hat{g}}(\Sigma_h, \partial^f B_h^\delta)=\measuredangle_{\hat{g}}(\partial^t B_h^\delta, \partial^f B_h^\delta)=\measuredangle_{\hat{g}}(\partial^s B_h^\delta, \partial^f B_h^\delta)=\frac{\pi}{2}$.
	\end{itemize}
	See Figure \ref{fig:domains} and Figure \ref{fig:half-cylinderwedge} for the shape of $B_h^\delta$ in the cases of half-cylinder and strip, respectively. Topologically, $B_h^\delta$ looks like a curved-prism with $\Sigma_h$ as the bottom face and $\partial^t B_h^\delta$ as the top face, and it has side faces $\partial^s B_h^\delta$, $\partial^f B_h^\delta$, and $\partial^e B_h^\delta$.
	The third condition ensures that $|\partial^s B_h^\delta|_{\hat{g}}$ is uniformly bounded from above, independent of $\delta$.
	Note that $\partial^e B_h^\delta$ is disjoint from any other boundary component of $B_h^\delta$.
	
	Now, we define $\mathcal{E}$ to be the collection of all (relatively) open subsets $\Omega$ of $\hat{M}$ with finite perimeter in $B_h^\delta$ such that following holds:
	\begin{itemize}
		\item $\Omega \cap \Sigma_h =\emptyset$,
		\item $\partial^t B_h^\delta \subset \Omega$.
	\end{itemize}
	For each $\Omega \in \mathcal{E}$, we define the $\theta$-energy $E_\theta^{\hat{g}}$ under the metric $\hat{g}$ of $\Omega$ as
	\[
		E_\theta^{\hat{g}}(\Omega)=|\partial \Omega \cap \mathring B_h^\delta|_{\hat{g}}-\cos \theta |\partial \Omega\cap \partial^s B_h^\delta|_{\hat{g}}.
	\]
    Apparently, $E_\theta^{\hat{g}}(\Omega)$ has a finite lower bound since
    \[E_\theta^{\hat{g}}(\Omega)\geq -|\cos\theta||\partial^sB_h^\delta|_{\hat{g}}>-\infty.\]
    
    We choose a smooth metric $\hat{g}(r,t,h)$, conformal to $\hat{g}(r,t)$, such that the following conditions hold:
	\begin{itemize}
		\item Under such metric, each face of $B_h^\delta$ is mean convex, and $\Sigma_h$ and $\partial^t B_h^\delta$ is strictly mean convex somewhere.
		\item $(1-\delta)\hat{g}(r,t)\le \hat{g}(r,t,h)\le (1+\delta)\hat{g}(r,t)$.
	\end{itemize}
	In general, $\hat{g}(r,t,h)$ can also be chosen to agree with $\hat{g}(r,t)$ away from a small neighborhood of $\partial B_h^\delta\backslash\partial^eB_h^\delta$.
	According to Li's existence and regularity result \cite{Li2019comparison}, there exists a non-trivial minimizer $\Omega(r,t,h,\delta)$ of $E_\theta^{\hat{g}}$ in $\mathcal{E}$. Furthermore, the surface $\Sigma(r,t,h,\delta):=\partial \Omega(r,t,h,\delta) \cap \mathring B_h^\delta$ is smooth away from the corners and at least $C^{1,\alpha}$ up to the corners of $B_h^\delta$. See Figure \ref{fig:minimizer} for the shape of $\Sigma(r,t,h,\delta)$ when $S$ is a strip.

	\begin{figure}[ht]
        \centering
	\begingroup
	\def\svgwidth{0.8\columnwidth}
\begingroup%
  \makeatletter%
  \providecommand\color[2][]{%
    \errmessage{(Inkscape) Color is used for the text in Inkscape, but the package 'color.sty' is not loaded}%
    \renewcommand\color[2][]{}%
  }%
  \providecommand\transparent[1]{%
    \errmessage{(Inkscape) Transparency is used (non-zero) for the text in Inkscape, but the package 'transparent.sty' is not loaded}%
    \renewcommand\transparent[1]{}%
  }%
  \providecommand\rotatebox[2]{#2}%
  \newcommand*\fsize{\dimexpr\f@size pt\relax}%
  \newcommand*\lineheight[1]{\fontsize{\fsize}{#1\fsize}\selectfont}%
  \ifx\svgwidth\undefined%
    \setlength{\unitlength}{706.20891421bp}%
    \ifx\svgscale\undefined%
      \relax%
    \else%
      \setlength{\unitlength}{\unitlength * \real{\svgscale}}%
    \fi%
  \else%
    \setlength{\unitlength}{\svgwidth}%
  \fi%
  \global\let\svgwidth\undefined%
  \global\let\svgscale\undefined%
  \makeatother%
  \begin{picture}(1,0.36535186)%
    \lineheight{1}%
    \setlength\tabcolsep{0pt}%
    \put(0,0){\includegraphics[width=\unitlength,page=1]{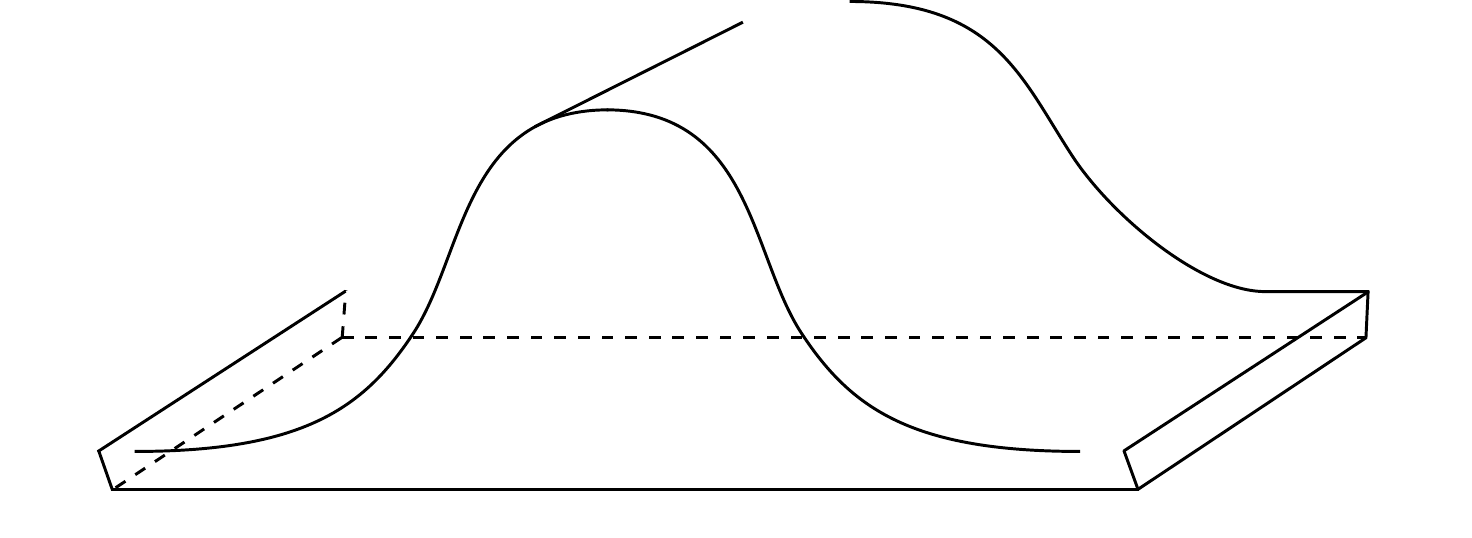}}%
    \put(0.38242052,-0.00073851){\makebox(0,0)[lt]{\lineheight{1.25}\smash{\begin{tabular}[t]{l}$\Sigma_h$\end{tabular}}}}%
    \put(0,0){\includegraphics[width=\unitlength,page=2]{domainsMinimizer.pdf}}%
    \put(0.51282191,0.30123575){\makebox(0,0)[lt]{\lineheight{1.25}\smash{\begin{tabular}[t]{l}$\partial^tB_h^\delta$\end{tabular}}}}%
    \put(0,0){\includegraphics[width=\unitlength,page=3]{domainsMinimizer.pdf}}%
    \put(0.35428191,0.05912148){\color[rgb]{1,0,0}\makebox(0,0)[lt]{\lineheight{1.25}\smash{\begin{tabular}[t]{l}$\Sigma(r,t,h,\delta)$\end{tabular}}}}%
  \end{picture}%
\endgroup%

	\endgroup

		\caption{The shape of $\Sigma(r,t,h,\delta)$ in the case of strip}
        \label{fig:minimizer}
    \end{figure}

	We claim that $\Sigma(r,t,h,\delta)$ intersects $\{x \in \hat{M} : \mathrm{dist}_{\hat{g}}(x, c(2r)) < 3r\}$ for sufficiently small $\delta=\delta(r,t,h)$.
	For simplicity, we denote $E_\theta^{\hat{g}}(B_h^\delta)=|\Sigma_h|_{\hat{g}}-\cos \theta |\partial^s B_h^\delta|_{\hat{g}}$.
	Since $\Sigma$ is absolutely $\theta$-energy-minimizing in $\hat{M}$ (in fact, being homologically$^*$ $\theta$-energy-minimizing is enough), let $\Omega':=B_h^\delta\backslash \Omega(r,t,h,\delta)$, we have
	\[
		|\Sigma_h|_{\hat{g}}\le |\Sigma(r,t,h,\delta)|_{\hat{g}}+|\partial \Omega' \cap \partial^f B_h^\delta|_{\hat{g}}+\cos \theta |\partial \Omega'\cap \partial^s B_h^\delta|_{\hat{g}}.
	\]
	Using the fact that $|\partial^f B_h^\delta|_{\hat{g}}<2\delta$, it follows that
	\[
		E_\theta^{\hat{g}}(B_h^\delta)\le E_\theta^{\hat{g}}(\Omega(r,t,h,\delta))+2\delta.
	\]
	Thus,
	\begin{align*}
		0<{} & E_{\theta}^{\hat{g}}(B_h^\delta)-E_{\theta}^{\hat{g}(r,t)}(B_h^\delta)\\ 
		\le& E_{\theta}^{\hat{g}}(\Omega(r,t,h,\delta))-E_{\theta}^{\hat{g}(r,t)}(B_h^\delta)+2\delta\\
        =&E_{\theta}^{\hat{g}(r,t)}(\Omega(r,t,h,\delta))-E_{\theta}^{\hat{g}(r,t)}(B_h^\delta)+2\delta\\
		\le{}& \frac{1}{1-\delta}|\Sigma(r,t,h,\delta)|_{\hat{g}(r,t,h)}-\frac{\cos \theta }{1-\delta}|\partial^s \Omega(r,t,h,\delta)|_{\hat{g}(r,t,h)}\\
			 &+\frac{2\delta |\cos \theta|}{1-\delta}|\partial^s  \Omega(r,t,h,\delta)|_{\hat{g}(r,t)} - E_{\theta}^{\hat{g}(r,t)}(B_h^\delta)+2\delta\\
		\le{}& \frac{1}{1-\delta}E_{\theta}^{\hat{g}(r,t,h)}(B_h^\delta)-E_{\theta}^{\hat{g}(r,t)}(B_h^\delta)+\frac{2\delta |\cos \theta|}{1-\delta}|\partial^s\Omega(r,t,h,\delta)|_{\hat{g}(r,t)}+2\delta\\
		\le{}&\frac{1+\delta}{1-\delta}E_{\theta}^{\hat{g}(r,t)}(B_h^\delta)-E_{\theta}^{\hat{g}(r,t)}(B_h^\delta)+\frac{2\delta |\cos \theta|}{1-\delta}|\partial^s \Omega(r,t,h,\delta)|_{\hat{g}(r,t)}\\
			 &+\frac{2\delta |\cos \theta|}{1-\delta}|\partial^s B_h^\delta|_{\hat{g}(r,t)}+2\delta\\
		\le{}&\frac{2\delta}{1-\delta}E_{\theta}^{\hat{g}(r,t)}(B_h^\delta)+\frac{4\delta |\cos \theta|}{1-\delta}|\partial^s B_h^\delta|_{\hat{g}(r,t)}+2\delta\\
		\le{}&\frac{2\delta}{1-\delta}|\Sigma_h|_{\hat{g}(r,t)}+\frac{6\delta|\cos \theta|}{1-\delta}|\partial^s B_h^\delta|_{\hat{g}(r,t)}+2\delta.
	\end{align*}
	Note that the right-hand side of the first line equals to $|\Sigma_h|_{\hat{g}}-|\Sigma_h|_{\hat{g}(r,t)}$, which is independent of $\delta$.
	Since $|\partial^s B_h^\delta|_{\hat{g}(r,t)}$ is bounded above independent of $\delta$, we can take $\delta$ sufficiently small to make the last line arbitrarily small, leading to a contradiction.

	Since $\Sigma(r,t,h,\delta(r,t,h))$ has a local area bound independent on $r,h,t$ by a comparison argument, we may pass to a subsequence of $\Sigma(r,t,h,\delta(r,t,h))$ as $h \rightarrow +\infty$, if necessary, to obtain a properly embedded limit surface $\Sigma(r,t)$ that is homologically$^*$ $\theta$-energy-minimizing under metric $\hat{g}(r,t)$ (See Proposition \ref{prop_convergence}).

When $r,t>0$ are very small, the surface $\Sigma(r,t)$ contains a curve $\gamma(r,t)$ that intersects $\{x \in \hat{M} : \mathrm{dist}_{\hat{g}}(x, c(2r)) < 3r \}$ and is close to the curve $\gamma$ on $\Sigma$. And the $\gamma(r,t)$ has a neighborhood in $\Sigma(r,t)$ that  converges smoothly to a ribbon or rectangular around $\gamma$ on $\Sigma$ depending whether $S$ is a half-cylinder or a strip (See \cite{chodoshsplittingtheorem}). Thus we can find a component $\hat{\Sigma}(r,t)$ of $\Sigma(r,t)$ and a curve $\gamma(r,t)$ within $\hat{\Sigma}(r,t)$ intersecting $\{ x \in \hat{M}: \mathrm{dist}_{\hat{g}}(x,c(2r))\le 3r \}$ which converges to $\gamma$ as $r,t \to 0^+$. Furthermore, $\gamma(r,t)$ can be a closed curve in $\hat{\Sigma}(r,t)$ when $S$ is a half-cylinder, and a compact curve in $\hat{\Sigma}(r,t)$ whose endpoints lie on $\partial \hat{\Sigma}(r,t)$ and it is perpendicular to $\partial \hat{\Sigma}(r,t)$ when $S$ is a strip.

	Now, let us determine the topology of $\hat{\Sigma}(r,t)$.
    \begin{lemma}
		\label{lem_topology}
		The topology of $\hat{\Sigma}(r,t)$ can be classified as follows when $r,t>0$ are sufficiently small:
		\begin{enumerate}[label=\textnormal{(\theenumi)}]
			\item If $S$ is a half-cylinder, then $\hat{\Sigma}(r,t)$ is either a half-cylinder or an annulus.
			\item If $S$ is a strip, and $S$ is absolutely $\theta$-energy-minimizing with its boundary lying on different boundary components of $M$,
				then $\hat{\Sigma}(r,t)$ is either a strip, or an annulus.
			\item If $S$ is a strip and is absolutely area-minimizing in the free boundary sense, then $\hat{\Sigma}(r,t)$ is either a strip, a half-cylinder, or an annulus.
			\end{enumerate}
	\end{lemma}
    \begin{proof}
	According to Lemma \ref{ambrozio}, we know $\chi(\Sigma)\ge 0$, which implies that $\Sigma$ is topologically a plane, a sphere, a torus, a cylinder, a half-plane, a disc, an annulus, a half-cylinder, or a strip.

	Note that by curvature estimate for capillary surfaces in \cite{hongcrelle} (see also \cite{liZhouZhu2024minMaxCapillary}), we know that $\hat{\Sigma}(r,t)$ converges to $\Sigma$ in any compact subset of $\hat{M}$, and $\partial \Sigma\neq \emptyset$, we have $\partial \hat{\Sigma}(r,t)\neq \emptyset$ for sufficiently small $r$ and $t$, and hence $\hat{\Sigma}(r,t)$ cannot be a plane, a sphere, a torus, or a cylinder.

	\vskip.2cm
    \noindent\textbf{Case: $S$ is a half-cylinder.}
	We claim that $\hat{\Sigma}(r,t)$ can only be an annulus or a half-cylinder.

	As $\hat{\Sigma}(r,t)$ converges to $\Sigma$ locally and $\Sigma$ has a compact boundary, we can rule out the possibility of $\hat{\Sigma}(r,t)$ being a half-plane, or a strip.

	We need the following lemma to exclude the possibility of $\hat{\Sigma}(r,t)$ being a disc, which is an easy adaptation of Lemma 3.2 in \cite{chodoshsplittingtheorem}.
	\begin{lemma}
		\label{lem:separate}
		If the inclusion $\Sigma \subset \hat{M}$ induces a trivial map $\pi_1(\Sigma)\rightarrow \pi_1(\hat{M})$, then every connected, compact surface $\Sigma'$ with $\partial \Sigma'\subset \partial^s \hat{M}$ in $\hat{M}$ is separating.
	\end{lemma}
        \begin{proof}
		It is easy to see that $M$ is simply-connected, since $i_*:\pi_1(\Sigma)\rightarrow \pi_1(M)$ is also surjective.
		We show that $\Sigma'$ separates $M$.
		If not, then we can find a simple closed curve $\gamma$ in $M$ such that it intersects $\Sigma'$ transversely at exactly one point.
		Hence $\Sigma'$ represents a nontrivial homology class in $H_{2}(M,\partial M)$, and $\gamma$ is not null-homotopic in $M$, which contradicts the fact that $M$ is simply-connected.
		Hence $\Sigma'$ separates $\hat{M}$ since $\Sigma' \cap \Sigma$ is empty.
	\end{proof}

	Now, we assume $\hat{\Sigma}(r_0,t_0)$ is a disc for some $r_0,t_0>0$.
	Since $\gamma(r_0,t_0)$ is null-homotopic, so is $\gamma$ for sufficiently small $r_0,t_0$.
	By Lemma \ref{lem:separate}, we know that $\hat{\Sigma}(r_0,t_0)$ is separating.
	Then, we can choose the region $\Omega_0$ such that $\partial^i \llbracket \Omega_0\rrbracket=\llbracket \Sigma\rrbracket-\llbracket \hat{\Sigma}(r_0,t_0)\rrbracket$. (Recall we have fixed the orientation of $\Sigma$ such that $\partial^i \llbracket \hat{M}\rrbracket=\llbracket \Sigma\rrbracket$.)
	Taking the boundary on the both sides, and using $\partial (\partial \llbracket \Omega_0\rrbracket \lfloor(\partial \hat{M}\backslash (\Sigma \cup \partial^s \hat{M})))=0$, we have $\partial\partial^s \llbracket \Omega_0\rrbracket=\partial \llbracket \hat{\Sigma}(r_0,t_0)\rrbracket-\partial \llbracket \Sigma\rrbracket$.
	Then, since $\hat{\Sigma}(r_0,t_0)$ converges to $\Sigma$ as $r_0,t_0\to 0^+$ locally, together with the constancy theorem (\cite[Theorem 26.27]{simon1983lectures}), we know the support of $\partial^s \llbracket \Omega_0\rrbracket$ is indeed the region bounded by $\partial \Sigma$ and $\partial \hat{\Sigma}(r_0,t_0)$ in $\partial^s \hat{M}$.
	Moreover, as $\hat{\Sigma}(r_0,t_0)$ converges to $\Sigma$ in the neighborhood of $\partial \Sigma$ smoothly, we can assume that the support of $\partial^s \llbracket \Omega_0\rrbracket$ is a bounded annular region by choosing $r_0,t_0$ small enough.
	Hence, $|\partial^s \Omega_0|_{\hat{g}}<+\infty$.
	Now, we can show that $\hat{\Sigma}(r,t)$ cannot be a disc for sufficiently small $0<r\ll r_0,0<t\ll t_0$.
	Again, we know that $\hat{\Sigma}(r,t)$ and $\Sigma$ bounds a region $\Omega$ in $\hat{M}$ with $\partial^i \llbracket \Omega\rrbracket=\llbracket \Sigma\rrbracket-\llbracket \hat{\Sigma}(r,t)\rrbracket$.
	Then, we have $\partial\partial^s \llbracket \Omega\rrbracket=\partial \llbracket \hat{\Sigma}(r,t)\rrbracket-\partial \llbracket \Sigma\rrbracket$.
	We consider the current $T=\llbracket \Omega\rrbracket -\llbracket \Omega_0\rrbracket$.
	Note that $\mathrm{spt}\partial^sT\subset \mathrm{spt}\partial^s\llbracket \Omega_0\rrbracket$ for sufficiently small $r,t>0$, we know $|\partial^sT|_{\hat{g}}<+\infty$.
	Since $\hat{\Sigma}(r,t)$ is homologically$^*$ $\theta$-energy-minimizing 
    by Proposition \ref{prop:homologically_energy_minimizing}, we have
	\[
		|\hat{\Sigma}(r,t)|_{\hat{g}(r,t)}\le |\hat{\Sigma}(r_0,t_0)|_{\hat{g}(r,t)}+\cos \theta|\partial^sT|_{\hat{g}}<+\infty.
	\]
	But the area of $\hat{\Sigma}(r,t)$ goes to infinity as $r,t\to 0^+$, which contradicts the above inequality.
	Therefore, $\hat{\Sigma}(r,t)$ cannot be a disc for sufficiently small $r,t>0$.

	Therefore, $\hat{\Sigma}(r,t)$ can only be a half-cylinder or an annulus.

\begin{figure}[ht]
    \centering
	\begingroup
	\def\svgwidth{0.8\columnwidth}
\begingroup%
  \makeatletter%
  \providecommand\color[2][]{%
    \errmessage{(Inkscape) Color is used for the text in Inkscape, but the package 'color.sty' is not loaded}%
    \renewcommand\color[2][]{}%
  }%
  \providecommand\transparent[1]{%
    \errmessage{(Inkscape) Transparency is used (non-zero) for the text in Inkscape, but the package 'transparent.sty' is not loaded}%
    \renewcommand\transparent[1]{}%
  }%
  \providecommand\rotatebox[2]{#2}%
  \newcommand*\fsize{\dimexpr\f@size pt\relax}%
  \newcommand*\lineheight[1]{\fontsize{\fsize}{#1\fsize}\selectfont}%
  \ifx\svgwidth\undefined%
    \setlength{\unitlength}{621.61332258bp}%
    \ifx\svgscale\undefined%
      \relax%
    \else%
      \setlength{\unitlength}{\unitlength * \real{\svgscale}}%
    \fi%
  \else%
    \setlength{\unitlength}{\svgwidth}%
  \fi%
  \global\let\svgwidth\undefined%
  \global\let\svgscale\undefined%
  \makeatother%
  \begin{picture}(1,0.27846349)%
    \lineheight{1}%
    \setlength\tabcolsep{0pt}%
    \put(0,0){\includegraphics[width=\unitlength,page=1]{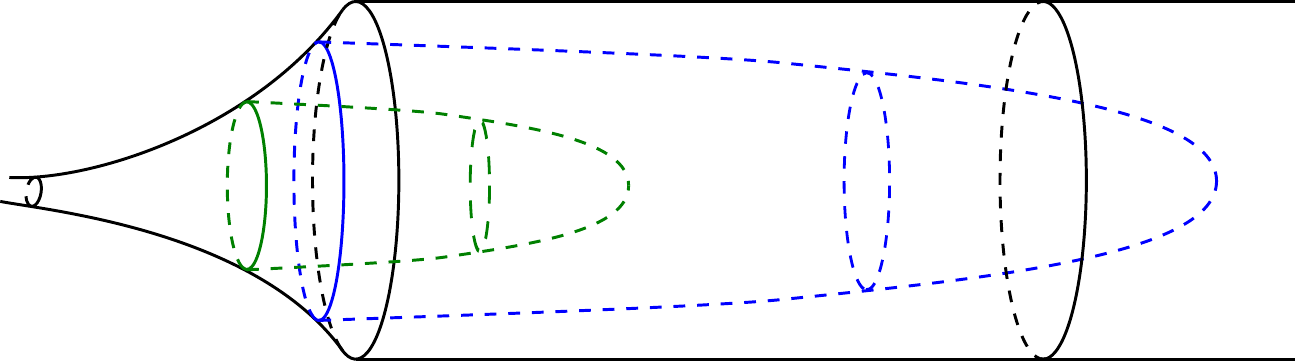}}%
    \put(0.9532334,0.12009437){\color[rgb]{0,0,1}\makebox(0,0)[lt]{\lineheight{1.25}\smash{\begin{tabular}[t]{l}$\hat{\Sigma}(r,t)$\end{tabular}}}}%
    \put(0.49597372,0.12655397){\color[rgb]{0,0.50196078,0}\makebox(0,0)[lt]{\lineheight{1.25}\smash{\begin{tabular}[t]{l}$\hat{\Sigma}(r_0,t_0)$\end{tabular}}}}%
    \put(0.86420944,0.24603214){\makebox(0,0)[lt]{\lineheight{1.25}\smash{\begin{tabular}[t]{l}$\Sigma$\end{tabular}}}}%
  \end{picture}%
\endgroup%

	\endgroup

	\caption{A disc converging to a half-cylinder}
    \label{fig:half-cylinder-disc}
\end{figure}
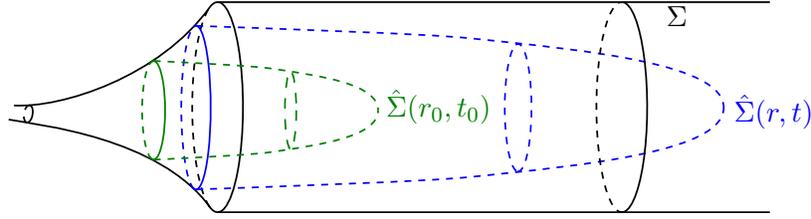
\vskip.2cm
	\noindent\textbf{Case: $S$ is a strip.}
	If we assume that each component of $\partial S$ lies on the different boundary components of $M$, then we can immediately obtain that $\hat{\Sigma}(r,t)$ is a strip or an annulus since the boundary of $\hat{\Sigma}(r,t)$ has at least two components as $r,t$ small enough. Thus, the second case is established.

	We now focus on the final case in the lemma.
	The key is to use the minimizing property to rule out the possibility of $\hat{\Sigma}(r,t)$ being a disc or a half-plane.
	It is not hard to see that $\hat{\Sigma}(r,t)$ is homologically$^*$ area-minimizing in the free boundary sense in $\hat{M}$.

	If $\hat{\Sigma}(r,t)$ is a disc, similar to the earlier case, together with Lemma \ref{lem:separate}, it follows that $\hat{\Sigma}(r,t)$ bounds a region in $\hat{M}$.
	This contradicts the fact that $\Sigma(r,t)$ is homologically$^*$ area-minimizing in the free boundary sense in $\hat{M}$.

	Now, suppose $\hat{\Sigma}(r,t)$ is a half-plane. Let $\Delta(r,t)$ denote the half-disk in $\hat{\Sigma}(r,t)$ bounded by $\gamma(r,t)$, which is the compact component of $\hat{\Sigma}(r,t) \setminus \gamma(r,t)$. As $\hat{\Sigma}(r,t)$ converges to $\Sigma$, the area $|\Delta(r,t)|_{\hat{g}}$ diverges to infinity as $r, t \to 0^+$. Choose parameters $0 < r_2 \ll r_1$ and $0 < t_2 \ll t_1$ such that $\hat{\Sigma}(r_1, t_1)$ and $\hat{\Sigma}(r_2, t_2)$ are both half-planes.
	Because $\gamma(r,t)$ converges to $\gamma$, we can find a disk $D$ (essentially a curved rectangle) in $\hat{M}$ such that $\partial D \cap \mathring{\hat{M}} = \gamma(r_1, t_1) \cup \gamma(r_2, t_2)$ and $|D|_{\hat{g}}$ is sufficiently small.
	See Figure \ref{fig:half-plane} for an illustration of $D$, $\Delta(r_1,t_1)$, and $\Delta(r_2,t_2)$.

	\begin{figure}[ht]
		\centering
	\begingroup
	\def\svgwidth{0.8\columnwidth}
\begingroup%
  \makeatletter%
  \providecommand\color[2][]{%
    \errmessage{(Inkscape) Color is used for the text in Inkscape, but the package 'color.sty' is not loaded}%
    \renewcommand\color[2][]{}%
  }%
  \providecommand\transparent[1]{%
    \errmessage{(Inkscape) Transparency is used (non-zero) for the text in Inkscape, but the package 'transparent.sty' is not loaded}%
    \renewcommand\transparent[1]{}%
  }%
  \providecommand\rotatebox[2]{#2}%
  \newcommand*\fsize{\dimexpr\f@size pt\relax}%
  \newcommand*\lineheight[1]{\fontsize{\fsize}{#1\fsize}\selectfont}%
  \ifx\svgwidth\undefined%
    \setlength{\unitlength}{637.37078809bp}%
    \ifx\svgscale\undefined%
      \relax%
    \else%
      \setlength{\unitlength}{\unitlength * \real{\svgscale}}%
    \fi%
  \else%
    \setlength{\unitlength}{\svgwidth}%
  \fi%
  \global\let\svgwidth\undefined%
  \global\let\svgscale\undefined%
  \makeatother%
  \begin{picture}(1,0.31914012)%
    \lineheight{1}%
    \setlength\tabcolsep{0pt}%
    \put(0,0){\includegraphics[width=\unitlength,page=1]{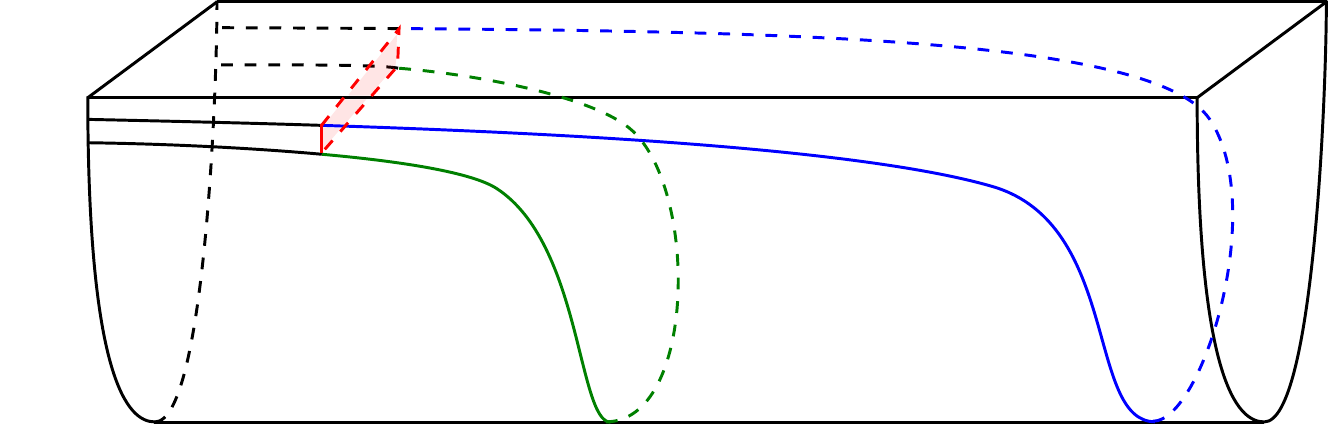}}%
    \put(0.26851675,0.06104535){\color[rgb]{0,0.50196078,0}\makebox(0,0)[lt]{\lineheight{1.25}\smash{\begin{tabular}[t]{l}$\Delta(r_1,t_1)$\end{tabular}}}}%
    \put(0.71408274,0.21282999){\color[rgb]{0,0,1}\makebox(0,0)[lt]{\lineheight{1.25}\smash{\begin{tabular}[t]{l}$\Delta(r_2,t_2)$\end{tabular}}}}%
    \put(0.2179451,0.16874059){\color[rgb]{1,0,0}\makebox(0,0)[lt]{\lineheight{1.25}\smash{\begin{tabular}[t]{l}$D$\end{tabular}}}}%
    \put(1.02480637,0.31550297){\makebox(0,0)[lt]{\lineheight{1.25}\smash{\begin{tabular}[t]{l}$\Sigma$\end{tabular}}}}%
    \put(-0.10502922,0.18469562){\makebox(0,0)[lt]{\lineheight{1.25}\smash{\begin{tabular}[t]{l}$\hat{\Sigma}(r_1,t_1)$\end{tabular}}}}%
    \put(-0.11134655,0.2382549){\makebox(0,0)[lt]{\lineheight{1.25}\smash{\begin{tabular}[t]{l}$\hat{\Sigma}(r_2,t_2)$\end{tabular}}}}%
    \put(0.60533862,0.06441361){\makebox(0,0)[lt]{\lineheight{1.25}\smash{\begin{tabular}[t]{l}$\Omega$\end{tabular}}}}%
  \end{picture}%
\endgroup%

	\endgroup

		\caption{Choice of $D$, $\Delta(r_1,t_1)$, and $\Delta(r_2,t_2)$}
		\label{fig:half-plane}
	\end{figure}
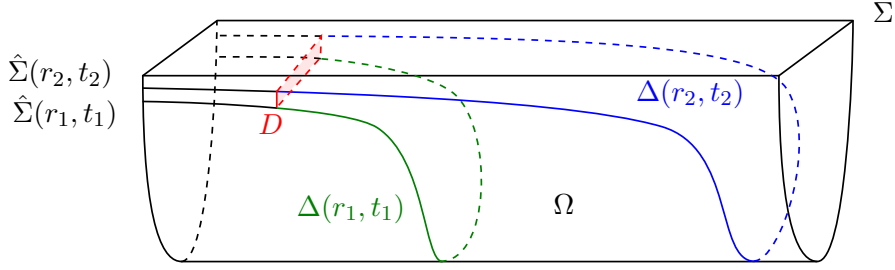

	Note that $\Delta(r_1,t_1)\cup \Delta(r_2,t_2)\cup D$ is homeomorphic to a disk.
	Then, by the previous arguments, together with Lemma \ref{lem:separate}, we know that $\Delta(r_1,t_1)\cup \Delta(r_2,t_2)\cup D$ bounds a possibly unbounded region $\Omega$ in $\hat{M}$, but $$|\Delta(r_2,t_2)|_{\hat{g}}>|\Delta(r_1,t_1)|_{\hat{g}}+|D|_{\hat{g}}$$for sufficiently small $r_2$ and $t_2$.
	This contradicts the fact that $\Sigma(r_2,t_2)$ is homologically$^*$ area-minimizing in the free boundary sense. Thus, $\hat{\Sigma}(r,t)$ cannot be a half-plane.

	In general, if $S$ is a strip, we cannot exclude the possibility that $\hat{\Sigma}(r,t)$ is a half-cylinder. See Figure \ref{fig:strip} for example.

    \begin{figure}[ht]
        \centering
	\begingroup
	\def\svgwidth{0.8\columnwidth}
	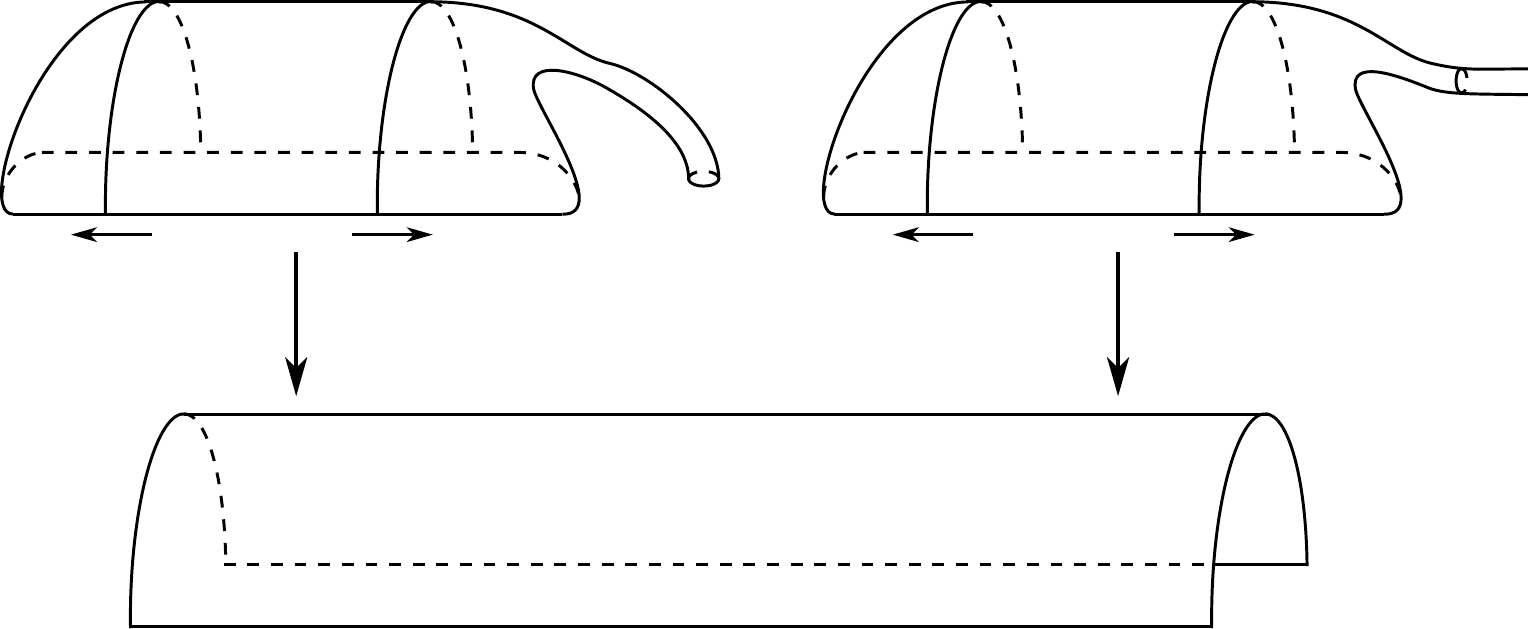
	\endgroup

    	\caption{An annulus or half-cylinder converging to a strip}
        \label{fig:strip}
    \end{figure}
\vskip.3cm

	In conclusion, $\hat{\Sigma}(r,t)$ must be a half-cylinder, annulus (see Figure \ref{fig:converge}), or strip (shows up only when $S$ is a strip).
    \end{proof}
    \begin{figure}[ht]
        \centering
	\begingroup
	\def\svgwidth{0.8\columnwidth}
	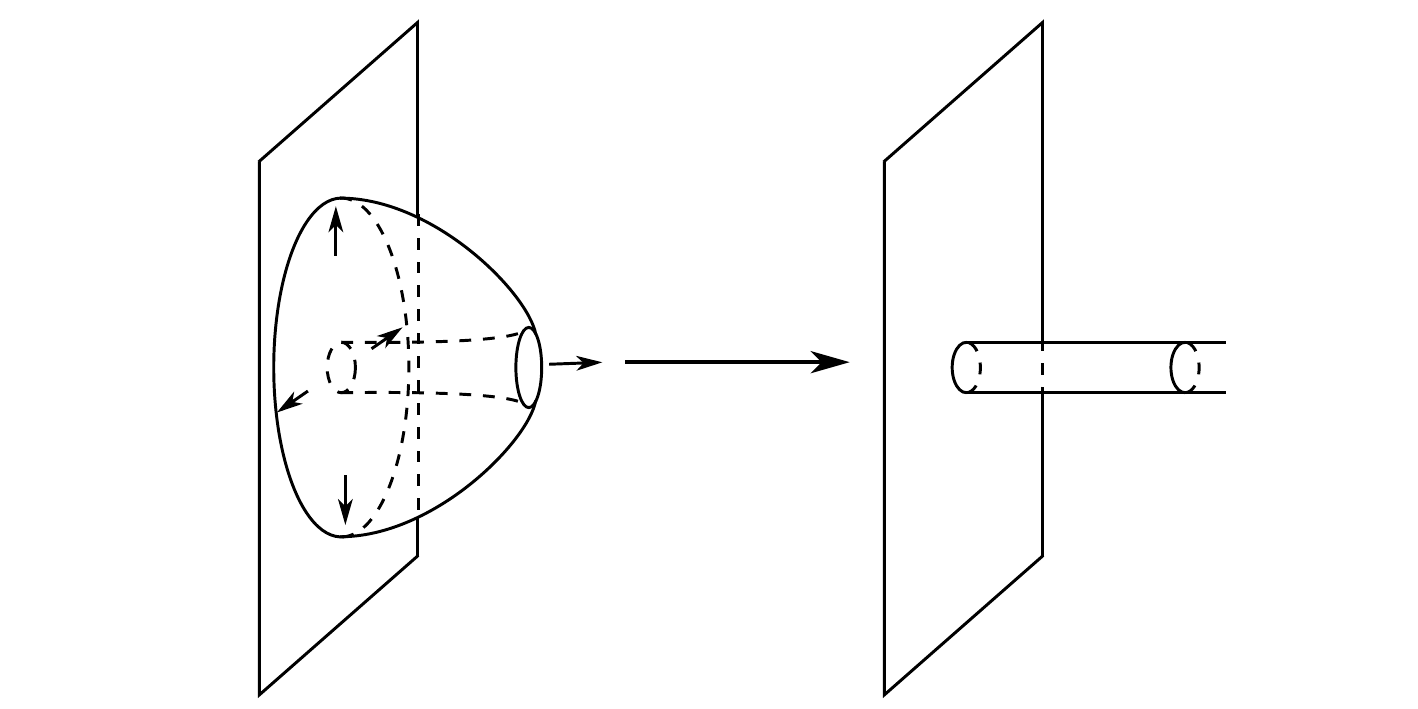
	\endgroup

        \caption{An annulus converging to a half-cylinder}
        \label{fig:converge}
    \end{figure}

    Let's go back to the proof of main theorem.
	By Lemma \ref{lemma:flatAnnulus} and Lemma \ref{lemma:flatnessBdry}, for $r,t>0$ sufficiently small, $\hat{\Sigma}(r,t)$ intersects either $\{x \in \hat{M} : \mathrm{dist}_g(x, c(2r)) = 3r\}$ or $\{x \in \hat{M} : \mathrm{dist}_g(x, c(2r)) <r \}$. Fixing small $r > 0$ and taking the limit as $t \to 0^+$, we obtain a properly embedded surface $\Sigma(r)$ as the limit of $\Sigma(r,t)$ after passing to a subsequence. Let $\hat{\Sigma}(r)$ denote the component of $\Sigma(r)$ intersecting $\{x \in \hat{M} : \mathrm{dist}_g(x, c(2r)) \leq  3r\}$.
	Since $\Sigma(t)$ is homologically$^*$ $\theta$-energy-minimizing in $\hat{M}$, the above argument implies that $\hat{\Sigma}(r)$ is either a half-cylinder, annulus, or strip (only when $S$ is a strip).

	Now, we consider the case that $S$ is a half-cylinder.
	If $\hat{\Sigma}(t)$ is an annulus, since $M$ is smooth manifold, then by Lemma \ref{lemma:flatAnnulus} and a continuation argument, we conclude that $\theta=\frac{\pi}{2}$ and $(\hat{M}, \hat{g})$ is isometric to either $([0, +\infty) \times \hat{\Sigma}(t), dt^2 + g_{\hat{\Sigma}(t)})$ or $([0, a] \times \hat{\Sigma}(t), dt^2 + g_{\hat{\Sigma}(t)})$ for some $a > 0$.

	Hence, we assume $\hat{\Sigma}(r)$ is a half-cylinder for all $r>0$.
	Lemma \ref{lemma:flatAnnulus} implies that $\hat{\Sigma}(r)$ is totally geodesic, flat, and that the ambient Ricci tensor vanishes in the normal direction of $\hat{\Sigma}(r)$ along $\hat{\Sigma}(r)$.
	Thus we obtain a sequence of totally geodesic intrinsically flat free boundary half-cylinders in $M$ along which the normal Ricci curvature vanishes and that converge to $\Sigma$ from one side.
Now, we want to write $\hat{\Sigma}(r)$ as a graph over $\Sigma$ and analyze the asymptotic behavior of such graph function.
Let $N$ be the unit normal vector field of $\Sigma$ in $\hat{M}$, and $X$ be a smooth tangential vector field along $\Sigma$ such that $X+N$ is tangential to $\partial M$ along $\partial \Sigma$.
	We can extend $X+N$ to a smooth vector field $V$ on a neighborhood of $\Sigma$ and keep $V$ tangential to $\partial M$ along $\partial M$.
	Let $\varphi(x,t)$ be the flow generated by $-V$.
We also select a sequence of pre-compact subset $U_i$ with $U_i\subset U_{i+1}$ and $\cup_i U_i=M$.
For each $U_i$, we can always choose $r_i>0$ small enough such that
$\hat{\Sigma}(r_i)\cap U_i$ can be written as $\{ \varphi(x,f_i(x)):x \in U_i' \}$ where $U_i'\subset \Sigma'$ and $f_i(t)$ is a positive smooth function defined on $U_i'$.
Then, we can define the (positive) limit of the scaled functions, i.e., $f=\lim_{i\rightarrow +\infty} \frac{f_i}{f_i(p)}$, where $p$ is a pre-chosen point on $\Sigma$.

Differentiating the second fundamental form, we find that $f$ satisfies
    \[(\nabla^2f)(X,Y)+\operatorname{R}_M(X,\nu,Y,\nu)=0\]
    for any tangent vectors on $\Sigma$.
    Taking the trace yields $\Delta f=0$. Moreover, differentiating the angle function shows that $\frac{\partial f}{\partial \eta}=0$, where $\eta$ is the unit outer normal vector of $\partial\Sigma$ in $\Sigma$.
	By extending the function $f$ to a positive function $\tilde{f}$ on $\mathbb{R}^2$, we obtain that $f$ is a positive constant on $\Sigma$. We see that $\operatorname{R}_M(X,\nu,Y,\nu)=0$ for any tangent vectors on $\Sigma$. The Codazzi equation and the Gauss equation imply that
$R_M(X,Y,Z,\nu)=0$ and $R_M(X,Y,Z,W)=0$. It follows that the ambient curvature
tensor vanishes along $\Sigma$.
    Repeating this argument with $\hat{\Sigma}(r)$ in place of $S$, we obtain a collection of half-cylinders $\mathcal{S}:=\left\{ \Sigma': \Sigma'\text{ is homologically$^*$ $\theta$-energy-minimizing}\right\}$.
Moreover, $\cup _{\Sigma' \in \mathcal{S}}\Sigma'$ is dense in $\hat{M}$.
This implies $\hat{M}$ is indeed flat and together with the contact angle condition, $\hat{M}$ is isometric to one of the following:
$$\left\{ (x_1,x_2,x_3)\in \mathbb{S}^1 \times \mathbb{R}^2:x_2\ge \cot \theta x_3\ \text{and} \ 0\le x_2\le b \right\}$$ for some $b>0$ or $$\left\{ (x_1,x_2,x_3)\in \mathbb{S}^1 \times \mathbb{R}^2:x_2\ge \cot \theta x_3\ \text{and}\ x_3\ge 0 \right\}.$$


	If $S$ is a strip, 
	we assume $\hat{\Sigma}(r)$ is a strip for all $r > 0$. If this is not the case, the splitting theorem for annuli or half-cylinders applies to derive the splitting result. Analogous reasoning shows that near $\Sigma$, $\hat{M}$ is isometric to $\{ (x_1,x_2,t)\in \mathbb{R}\times \mathbb{R} \times [0,\varepsilon):x_1\in [\cot \theta t, a-\cot \theta t] \}$.
	However, using the continuation argument, we find $M$ has singular boundary components if $\theta\neq \frac{\pi}{2}$.
	Consequenctly, $\theta$ must be $\frac{\pi}{2}$, and $\hat{M}$ is isometric to either $[0,a]\times \mathbb{R}^+\times [0,b]$ for some $b>0$ or $[0,a]\times \mathbb{R}\times \mathbb{R}^+$.
\end{proof}

\bibliographystyle{alpha}
\bibliography{mybib}
\end{document}